\documentclass[a4paper,11pt]{amsart}
\usepackage[utf8]{inputenc}
\usepackage[T1]{fontenc}
\usepackage{lmodern}
\usepackage[utf8]{inputenc}
\usepackage{float}
\usepackage[english]{babel}
\usepackage{geometry} 
\usepackage{adjustbox}
\usepackage{booktabs} 
\usepackage{color}
\usepackage{amsthm}
\usepackage{amsmath,amssymb}
\usepackage{graphicx}
\usepackage{listings}
\usepackage{emptypage}
\usepackage{newlfont}
\usepackage{cleveref}

\usepackage{ stmaryrd }
\theoremstyle{plain} 

\theoremstyle{definition} 

\theoremstyle{remark} 

\newcommand{\ie}{\emph{i.e.}\ }

\numberwithin{equation}{section}
\newtheorem{theorem}[subsection]{Theorem}
\newtheorem*{theorem*}{Theorem}
\newtheorem{definition}[subsection]{Definition}
\newtheorem{decomposition-theorem}[subsection]{Decomposition Theorem}
\newtheorem{proposition-definition}[subsection]{Proposition-Definition}
\newtheorem{lemma}[subsection]{Lemma}
\newtheorem*{lemma*}{Lemma}
\newtheorem{proposition}[subsection]{Proposition}

\newtheorem{remark}[subsection]{Remark}
\newcommand{\reminder}[1]{}

\renewcommand{\mod}{\mathrm{mod}\,}

\newcommand{\proj}{\mathrm{proj}\,}

\newcommand{\ind}{\mathrm{ind}}

\newcommand{\op}{^{op}}

\newcommand{\Z}{\mathbb{Z}}

\newcommand{\mL}{\mathrm{L}}

\newcommand{\Hom}{\mathrm{Hom}}

\newcommand{\Ext}{\mathrm{Ext}}

\newcommand{\cp}{{\mathcal P}}

\newcommand{\ct}{{\mathcal T}}

\renewcommand{\phi}{\varphi}

\renewcommand{\tilde}[1]{\widetilde{#1}}

\usepackage{setspace}
\usepackage[utf8]{inputenc} 

\usepackage{graphicx}
\usepackage{tikz-cd}

\usepackage[all]{xy}
\usepackage{epsf}

\usepackage{graphicx,amssymb}

\date{October 28, 2021, last modified on \today}

\title{Equivalent definitions of the preprojective algebra}
\author{
Chiara Sava}
\email{
sava@karlin.mff.cuni.cz}

\begin{document}

\begin{abstract}
Following the article of C. M. Ringel we introduce preprojective algebras of a Dynkin quiver $Q$ starting from three definitions which, despite concerning completely different algebraic structures, turn out to be equivalent. Our main result is a new version of Ringel's proofs that applies a theorem by Happel and exploits the techniques of homological algebra.
Moreover we show that the  definition of the preprojective algebra given with the usual notion of commutator is equivalent to the definition with the "generalised" commutator. 
\end{abstract}

\maketitle

\tableofcontents
\section{Introduction} 
Let $k$ be an algebraically closed field and fix a Dynkin quiver $Q$. To study the preprojective algebra of $Q$ we can choose between different approaches: the first one concerns a quotient of a $k$-algebra by an ideal and leads to the \emph{combinatorial definition} (\Cref{combinatorial}); the second one concerns the structure of the morphisms in the orbit category and leads to the \emph{homological definition} (\Cref{hodef}); finally, the third one is based on a particular tensor algebra and then leads to the \emph{tensor definition} (\Cref{tedef}). One of the basic results about preprojective algebras shows that these three approaches are equivalent.

In the article \cite{Ringel97thepreprojective}, C. M. Ringel first proves what he calls Theorem A, which states the equivalence between the combinatorial and the tensor definition. Then he proves that the tensor definition is also equivalent to the homological one in Theorem B.

In this article we give an alternative proof of these equivalences following a different strategy. In particular, the result by Happel (\Cref{Happel}) which describes the derived category $\ind D^b(\mod kQ)$ as being isomorphic to the mesh category, provides the link between the underlying structures of the quotient algebra and the morphism space in the orbit category. Then we prove Theorem B with homological algebra techniques using \Cref{keylemma} to switch from a derived tensor product to a tensor product which is no longer derived.

Furthermore, we prove that there is another equivalent way how to define the preprojective algebra related to the combinatorial definition. In particular, the ideal we mentioned above is generated by commutators that are defined in the usual way as $[\alpha, \alpha^*]=\alpha^*\alpha-\alpha\alpha^*$. In \Cref{commutator}
we consider a "generalised" version of them, namely \[[\alpha, \alpha^*]_{q(\alpha)}=\alpha^*\alpha-q(\alpha)\alpha\alpha^*,\] where $q$ is a function from the set of arrows $Q_1$ to the field $k^*=k \setminus \{0\}$. It turns out that the standard combinatorial definition is equivalent to the one given by means of generalised commutators.  
\section*{Acknowledgement}
This paper originated from the author's Master thesis. The author would like to thank Bernhard Keller for his supervision, fundamental support and explaining the key ideas behind this work. The author also thanks Jan Šťovíček for his valuable comments on a draft of the paper and Yilin Wu for interesting discussions.

\section{Combinatorial definition}

Let $k$ be an algebraically closed field and let $Q$ be a Dynkin quiver. We denote the set of vertices by $Q_0$, the set of arrows by $Q_1$ and the two functions, respectively the source and the target of each arrow by  $s,t:Q_1 \to Q_0$.

\begin{definition}\label{combinatorial}
The \emph{doubled quiver} $\overline{Q}$ is obtained from $Q$ by adding for every arrow $\alpha:i \to j$ a new arrow $\alpha^*:j \to i$. The set of new arrows will be denoted by $Q_1^*$.
\end{definition}

We define the following element of the path algebra $k\overline{Q}$:
\[ \rho=\sum_{\alpha \in Q_1} [\alpha, \alpha^*],\]
where the summands $[\alpha, \alpha^*]=\alpha^*\alpha-\alpha\alpha^*$ are the usual commutators. 
\begin{remark}
We can see the combinatorial nature of this element through the following relations
\[ \rho =\sum_{\alpha \in Q_1} [\alpha, \alpha^*]=\sum_{\alpha \in Q_1}\alpha^*\alpha-\alpha\alpha^*=\sum_{i \in Q_0}(\sum_{t(\alpha)=i} \alpha\alpha^* - \sum_{s(\beta)=i} \beta^*\beta) \]
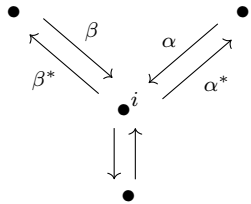
\[\begin{tikzcd}
\bullet \arrow[rd, "\beta", shift left=2] &                                                                                                             & \bullet \arrow[ld, "\alpha"', shift right] \\
                                          & \bullet^i \arrow[ru, "\alpha^*"', shift right=2] \arrow[lu, "\beta^*", shift left] \arrow[d, shift right=2] &                                            \\
                                          & \bullet \arrow[u, shift right]                                                                              &                                           
\end{tikzcd}\]
If we call $\rho_i=\sum_{t(\alpha)=i} \alpha\alpha^* - \sum_{s(\beta)=i} \beta^*\beta$, then we can also note that $\rho_i=e_i\rho e_i$, where $e_i$ is the idempotent element associated to the vertex $i$ of the path algebra. Thus, we have $\rho= \sum_{i\in Q_0} \rho_i=\sum_{i\in Q_0}e_i\rho e_i$ and we can conclude that, if we consider the two-sided ideal generated by $\rho$, then it is the same ideal as the one generated by the $\rho_i$'s: \[\langle \rho \rangle= \langle \rho_i | i \in Q_0 \rangle.\] 
\end{remark}

\begin{definition}[\textbf{Combinatorial Definition}]\label{codef}
The \emph{combinatorial preprojective algebra} of the quiver $Q$ is the quotient \[\Lambda^{co}_Q=k\overline{Q}/\langle \rho \rangle,\] where $\overline{Q}$ is the doubled quiver.
\end{definition}

\begin{remark}
We can generalize the concept of commutator, defining \[[\alpha, \alpha^*]_{q(\alpha)}=\alpha^*\alpha-q(\alpha)\alpha\alpha^*,\] where $q(\alpha)$ is an element of the field $k$. Thus, given an arbitrary function $q:Q_1 \to k$ we can consider \[\Lambda^{co}_{Q,q}=k\overline{Q}/\langle \rho_q \rangle,\quad \text{  where  }\quad \rho_q =\sum_{\alpha \in Q_1} [\alpha, \alpha^*]_{q(\alpha)}.\] 
\end{remark}

The next lemma holds not only in the Dynkin case but for all the quivers whose underlying graph is a tree. Then, we prove it in this more general setting. 
\begin{lemma}\label{commutator}
If the underlying graph of $Q$ is a tree and $q:Q_1 \to k^*$ is a function such that $q(\alpha)\ne 0$ for each arrow $\alpha$, then we have an isomorphism \[ \Lambda^{co}_{Q,q} \stackrel{\sim}{\rightarrow} \Lambda^{co}_{Q,1}.\]
\end{lemma}

In particular, if we look at the special case where $q$ is the constant function with value $-1$, then the lemma shows that, if $Q$ is a tree, the algebras $\Lambda^{co}_{Q}=\Lambda^{co}_{Q,1}$ and $\Lambda^{co}_{Q,-1}$ are isomorphic.

\begin{proof}
We will construct a function $\epsilon:Q_0 \to k^*$ such that the algebra isomorphism \[\varphi_{\epsilon}:k\overline{Q} \to k\overline{Q}\] taking an arrow $\alpha: i \to j$ to $\epsilon(i)^{-1}\epsilon(j)\alpha$ and $\alpha^*:j \to i$ to $\alpha^*$ induces an isomorphism
\[ \overline{\varphi_{\epsilon}}: \Lambda^{co}_{Q,q} \stackrel{\sim}{\rightarrow} \Lambda^{co}_{Q,1}.\]
We will construct $\epsilon$ recursively. For this, we first define a height function $h:Q_0 \to \Z$ such that 
\[h(j)=h(i)+1\] whenever there is an arrow $\alpha:i \to j $ in $Q$ and \[ \mathrm{min}_{i \in Q_0} h(i)=0.\]
For example for the quiver 
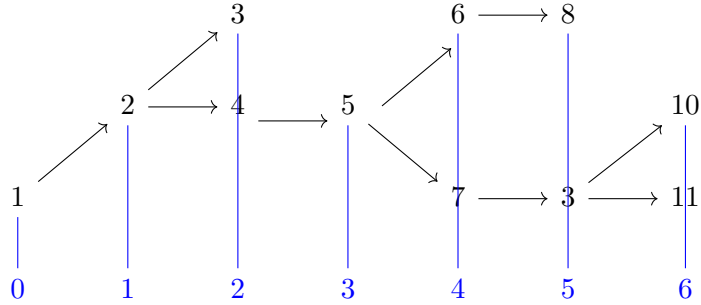
\begin{center}
    
\begin{tikzcd}
                                &                                            & 3 \arrow[ddd, no head, blue]     &                                                            & 6 \arrow[r] \arrow[ddd, no head, blue] & 8 \arrow[ddd, no head, blue] &                        \\
                                & 2 \arrow[ru] \arrow[r] \arrow[dd, no head, blue] & 4 \arrow[r, shift right=2] & 5 \arrow[ru, shift right=3] \arrow[rd] \arrow[dd, no head, blue] &                                  &                        & 10 \arrow[dd, no head, blue] \\
1 \arrow[ru] \arrow[d, no head, blue] &                                            &                            &                                                            & 7 \arrow[r]                      & 3 \arrow[ru] \arrow[r] & 11                     \\
\textcolor{blue}{0}             & \textcolor{blue}{1}                        & \textcolor{blue}{2}        & \textcolor{blue}{3}                                        & \textcolor{blue}{4}              & \textcolor{blue}{5}    & \textcolor{blue}{6}   
\end{tikzcd}
\end{center}
we have the height function \textcolor{blue}{h}.\\
For $n \ge 0$, denote by $Q^{\le n}$ the full subquiver of $Q$ whose vertices are the $i \in Q_0$ such that $h(i) \le n$. 
By induction on $n \ge 0$, we will define
\[\epsilon^{\le n}: Q_0^{\le n} \to k \]
such that the associated isomorphism
\[\varphi_{\epsilon^{\le n}}:k\overline{Q^{\le n}} \stackrel{\sim}{\longrightarrow} k\overline{Q^{\le n}}\]
defined like $\varphi_{\epsilon}$ above induces an isomorphism 
\[  \Lambda^{co}_{Q^{\le n},q^{\le n}} \stackrel{\sim}{\rightarrow} \Lambda^{co}_{Q^{\le n},1}\]
where $q^{\le n}$ is the restriction of $q$ to $Q^{\le n}_1$. 
For $n=0$, the quiver $Q^{\le 0}$ has no arrows and we define 
\[\epsilon^{\le 0}: Q_0^{\le 0} \to k \] to be constant with value 1. Now suppose that \[\epsilon^{\le n}: Q_0^{\le n} \to k \] is defined. Let $Q_0^{= n+1}$ be the set of vertices $i$ of height $h(i)=n+1$. Let $i \in Q_0^{=n+1}$. Let us denote the incoming and outgoing arrows of $Q$ at $i$ by $\alpha_1,\cdots,\alpha_s$ and $\beta_1,\cdots,\beta_t$ so that \[\alpha_i:s(\alpha_i) \to i, \quad \beta_j:i \to t(\beta_j).\]
The preprojective relation at $i$ in $Q^{\le n+1}$ is 
\begin{equation}\label{eq1}
    \sum_{u=i}^s q_{\alpha_u}\alpha_u\alpha^*_u - \sum_{v=i}^t \beta_v\beta_v^*
\end{equation}
respectively
\begin{equation}\label{eq2}
    \sum_{u=i}^s \alpha_u\alpha^*_u - \sum_{v=i}^t \beta_v\beta_v^*.
\end{equation}
Suppose we have chosen a function $\epsilon^{\le n+1}$ on $Q_0^{\le n+1}$. Then $\varphi_{\epsilon^{\le n+1}}$ transforms the equation \ref{eq1} into 
\begin{equation}\label{eq1.2}
    \sum_{u=i}^s \epsilon^{\bullet}(i)\epsilon^{\bullet}(s(\alpha_i))^{-1}q_{\alpha_u}\alpha_u\alpha^*_u - \sum_{v=i}^t\epsilon^{\bullet}(t(\beta_i))\epsilon^{\bullet}(i)^{-1} \beta_v\beta_v^*
\end{equation}
where $\epsilon^{\bullet}(i) := \epsilon^{\le n+1}(i)$. We now define 
\[\epsilon^{\bullet}(i)=1, \quad \epsilon^{\bullet}(s(\alpha_u))=q_{\alpha_u}, \quad \epsilon^{\bullet}(t(\beta_v))=1\]
for all $i$ and $j$ and \begin{equation}\label{eq3}
\epsilon^{\bullet}(k)=\epsilon^{\le n}(k)\epsilon^{\le n}(s(\alpha_u))^{-1}q_{\alpha_u}    
\end{equation}
for all vertices $k$ in the connected components of $s(\alpha_u)$ in $Q^{\le n}$. For different vertices $s(\alpha_u)$ the connected components of $s(\alpha_u)$ in $Q^{\le n}$ are disjoint because $Q$ is a tree. Thus, the definition \ref{eq3} does not lead to contradictions. Notice that on each connected components $C$ of $Q^{\le n}$ the new function $\epsilon^{\le n+1}$ restricts to a scalar multiple of the previous function $\epsilon^{\le n}\vert_{C}$. Therefore, the morphism
\[\varphi^{\le n+1}:k\overline{Q^{\le n+1}} \stackrel{\sim}{\longrightarrow} k\overline{Q^{\le n+1}}\]
takes the $q$-preprojective relations at all vertices $k$ of height $\le n$ to scalar multiples of the $1$-preprojective relation at $k$. 
By construction, the same holds for the vertices of height $n+1$. We now define $\epsilon$ to be $\epsilon^{\le n}$ where $n$ is any integer grater than or equal to the maximal height of a vertex of $Q$.

\end{proof}
\section{Homological definition}

\begin{definition}
Let $\ct$ be a $k$-linear triangulated category and $F: \ct \to \ct$ an autoequivalence that we can assume to be an automorphism. Let $F^\Z$ denote the group of automorphisms generated by $F$. By definition, the \emph{orbit category} $\ct/F=\ct/F^\Z$ has the same objects as $\ct$ and its morphisms from $X$ to $Y$ are in bijection with \[\bigoplus_{n \in \Z} \Hom_{\ct}(X,F^nY),\] \ie a morphism $u:X \to Y$ is given by a morphism $X \to \bigoplus_{i=1}^N F^{n_i}Y$ with $N$ non vanishing components $u_1,\cdots,u_N \in \ct$.
The composition is defined in the natural way.
\end{definition}

\begin{remark} Let $Q$ be a Dynkin quiver,
we recall that the Serre functor in the derived category $D^b(\mod kQ)$ has the following form: \[S=-\otimes^L_{kQ} D(kQ).\]
Moreover, we can explicitly compute its inverse. 
\begin{proposition}\label{S-1}
The inverse of the Serre functor $S^{-1}$ in $D^b(\mod A)$ is isomorphic to  \[-\otimes_A^L \mathrm{R}\Hom_A(DA,A).\]
\end{proposition}
\begin{proof}
We need the following lemma.
\begin{lemma}
Let $A,B$ be two algebras, if $X \in D(\mod A\op \otimes B)$ is a complex of bimodules such that $X\vert_B$ is perfect, then the right adjoint of $-\otimes_{A}^{\mathrm{L}}X$ is $-\otimes^{\mathrm{L}}_B \mathrm{R}\Hom_B(X,B)$. In particular, if $F=-\otimes_{A}^{\mathrm{L}}X$ is an equivalence then $X\vert_B= A \otimes^{\mathrm{L}}_A X \in D^b(\mod B)$ is perfect (since $A$ is perfect in $D(\mod A)$ and $F$ is an equivalence) and $F^{-1}$ is isomorphic to $-\otimes^{\mathrm{L}}_B\mathrm{R}\Hom_B(X,B)$.
\end{lemma}
Then $S^{-1}\simeq -\otimes_A^L \mathrm{R}\Hom_A(DA,A).$
\end{proof}
We also want to recall that the Auslander-Reiten translation $\tau$ can be expressed with the Serre functor in the following way:
\[\tau= [-1]\circ S,\]
where $[-1]$ is the shift functor.
\end{remark}

We consider the orbit category $D^b(\mod kQ)/\tau^-$.

\begin{definition} [\textbf{Homological definition}]\label{hodef}
The \emph{homological preprojective algebra} of the quiver $Q$ is the morphism space in the orbit category $D^b(\mod kQ)/\tau^-$. Then, it is of the form \[ \Lambda^{ho}_Q= \bigoplus_{i,j \in Q_0}\bigoplus_{p\ge0}\Hom_{D^b(\mod kQ)}(P_i, \tau^{-p}P_j),\]

where, if we denote by $e_i$ the idempotent element of the path algebra $kQ$ associated to the vertex $i$, then $P_i=e_ikQ$ is the indecompsable projective module associated to the vertex $i$.

\end{definition}

\section{Tensor Definition}

\begin{definition}
Let $\Lambda$ be a ring and $\Omega$ a $\Lambda$-bimodule. We put $\Omega^{\otimes 0}=\Lambda$ and for $n \ge 1$, we denote by \[\Omega^{\otimes n}=\Omega\otimes_{\Lambda}\Omega\otimes_{\Lambda}\otimes_{\Lambda}\cdots\otimes_{\Lambda}\Omega\]
the $n$th tensor power of $\Omega$ over $\Lambda$.  The \emph{tensor algebra} of $\Omega$ over $\Lambda$ is  \[T_\Lambda(\Omega)= \bigoplus_{t\ge 0} \Omega^{\otimes t}.\]

The product of $a \in \Omega^{\otimes s}$ and $b \in \Omega^{\otimes t}$ in the tensor algebra is given by the element $a \otimes b \in \Omega^{\otimes s+t}=\Omega^{\otimes s}\otimes \Omega^{\otimes t}$ with $s$ or $t \ge 1$ and the product $ab$ in $\Lambda$ otherwise.
\end{definition}

\begin{remark}
The ring $\Lambda$ is in particular a subring of $T_\Lambda(\Omega)$. Thus, there is a forgetful functor from the category of all $T_\Lambda(\Omega) $-modules to the category of $\Lambda$-modules.
\end{remark}

Let now $D$ be the usual duality $D=\Hom_k(-,k)$. We consider \[\Omega=\Ext^1_{kQ}(D(kQ), kQ),\]
where the right $kQ$-module structure comes from the left $kQ$-module structure on $D(kQ)$ and the left $kQ$-module structure from the left $kQ$-module structure on the second argument $kQ$. 

\begin{definition}[\textbf{Tensor definition}]\label{tedef}
Let $\Omega=\Ext^1_{kQ}(D(kQ), kQ)$. The \emph{tensor preprojective algebra} of the quiver $Q$ is the tensor algebra \[\Lambda^{te}_Q=T_{kQ}(\Omega).\] 
\end{definition}

\section{Equivalence between the three definitions}

We denote by $\cp Q$ the path category associated a quiver $Q$ and by $k \cp Q$ the $k$-linearization of the path category.\\
Consider a translation quiver ($\Z Q, \tau)$ and a map $\sigma: (\Z Q)_1 \to (\Z Q)_1$ such that, for any $\alpha:i \to j$, $\sigma(\alpha): \tau j \to i$.

 \begin{remark}
 For each vertex we then obtain a bijection
 \[ \sigma:\{\text{arrows ending in  }x\}\stackrel{\sim}{\to} \{\text{arrows starting in  }\tau(x)\}\]
 \end{remark}
 
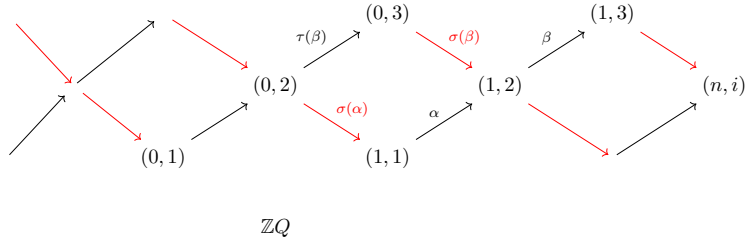
\begin{center}\adjustbox{scale=0.7,center}{
\begin{tikzcd}
{} \arrow[rd, red] &                          & {} \arrow[rd, red]      &                               & {(0,3)} \arrow[rd, "\sigma(\beta)", red] &                               & {(1,3)} \arrow[rd, red] &                                \\
              & {} \arrow[ru] \arrow[rd,  red] &                    & {(0,2)} \arrow[ru, "\tau(\beta)"] \arrow[rd,"\sigma(\alpha)", red] &                    & {(1,2)} \arrow[ru, "\beta"] \arrow[rd, red] &                    & {(n,i)}\\ 
{} \arrow[ru] &                          & {(0,1)} \arrow[ru] &                               & {(1,1)} \arrow[ru, "\alpha"] &                               & {} \arrow[ru]      &                               \\
              &                          &                    & \Z Q                           &                    &                               &                    &                                       
\end{tikzcd}}
 \end{center}

\begin{definition}
 The \emph{mesh ideal} in $\cp(\Z Q)$, is the bilateral ideal generated by the elements \[m_{(n,i)}=\sum_{\alpha \text{ s.t. } t(\alpha)=(n,i)} \alpha \circ \sigma(\alpha).\]
    \item The \emph{mesh category } $k(\Z Q)$ is the quotient category of the path category $\cp(\Z Q)$ by the mesh ideal. 
    \end{definition}
\begin{remark}
Given a quiver $Q$, consider the indecomposable projective right $kQ$-module associated to the vertex $i \in Q_0$, $P_i$. Thus, we have an equivalence of categories
\begin{center}
\begin{align*}
\hspace{4cm}& k \cp Q & &\stackrel{\sim}{\longrightarrow} & &\ind (\proj(kQ)) &\hspace{3cm}\\
\hspace{4cm}&i  & &\mapsto & &P_i &\hspace{3cm}\\
\hspace{4cm}&\alpha: i \to j & &\mapsto & &P_i \to P_j &\hspace{3cm}\\
\hspace{4cm}&\quad& &\quad & &p \mapsto \alpha p &\hspace{3cm}
\end{align*}
\end{center}

\end{remark}

\begin{theorem}[Happel]\label{Happel}\cite{Happel88}
Let $Q$ be a Dynkin quiver. Then we have the following equivalence of categories\[k(\Z Q)\simeq \ind D^b(kQ),\]
where we shortly write $D^b(kQ)$ instead of $D^b(\mod kQ)$.

In other words we have a commutative diagram

\[\begin{tikzcd}
i \arrow[rr, maps to]                                &  & P_i                            \\
k\cp Q \arrow[d, hook] \arrow[rr] \arrow[rr, "\sim"] &  & \ind (\proj(kQ)) \arrow[d, hook] \\
k(\Z Q) \arrow[rr] \arrow[rr, "\sim"]              &  & \ind D^b(kQ)                   \\
{(n,i)} \arrow[rr, maps to]                          &  & \tau^{-n}P_i                  
\end{tikzcd}\]

where $\tau$ is the Auslander-Reiten traslation.
\end{theorem}

We are now ready to to prove our main result. 
\begin{theorem}\label{equivalence}
We have canonical isomorphisms
\[\Lambda_Q^{co} \stackrel{\sim}{\longleftarrow} \Lambda_Q^{ho} \stackrel{\sim}{\longrightarrow} \Lambda_Q^{te}\] between the combinatorial, the homological and tensor preprojective algebra. 

\end{theorem}
\begin{remark}
Thus, the three definitions \Cref{codef}, \Cref{hodef} and \Cref{tedef} are equivalent. Up to isomorphism, they yield the same preprojective algebra $\Lambda_Q$.
\end{remark}
\begin{proof}

  \textit{First step: isomorphism between $\Lambda^{co}_Q$ and $\Lambda^{ho}_Q$.}
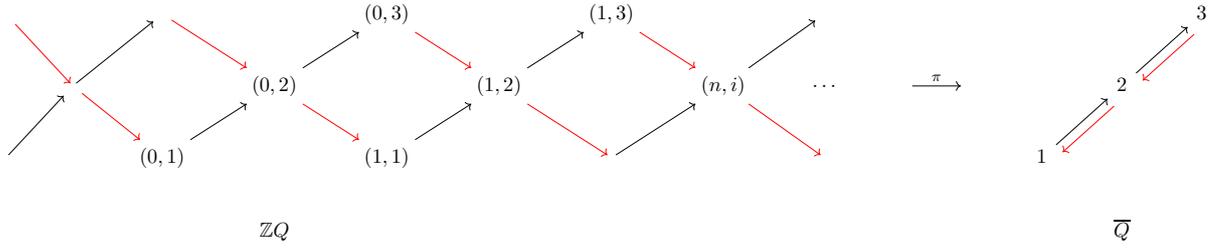
  \begin{center}
\adjustbox{scale=0.7}{\begin{tikzcd}
{} \arrow[rd, red] &                          & {} \arrow[rd, red]      &                               & {(0,3)} \arrow[rd, red] &                               & {(1,3)} \arrow[rd, red] &                               & {}     &                     &    &              &                                       & 3 \arrow[ld,red, shift left=2] \\
              & {} \arrow[ru] \arrow[rd, red] &                    & {(0,2)} \arrow[ru] \arrow[rd, red] &                    & {(1,2)} \arrow[ru] \arrow[rd, red] &                    & {(n,i)} \arrow[ru] \arrow[rd, red] & \cdots & {} \arrow[r, "\pi"] & {} &              & 2 \arrow[ru] \arrow[ld, red, shift left=2] &                            \\
{} \arrow[ru] &                          & {(0,1)} \arrow[ru] &                               & {(1,1)} \arrow[ru] &                               & {} \arrow[ru]      &                               & {}     &                     &    & 1 \arrow[ru] &                                       &                            \\
              &                          &                    & \Z Q                           &                    &                               &                    &                               &        &                     &    &              & \overline{Q}                          &                           
\end{tikzcd}}
 \end{center}
Let $\pi: \Z Q \to \overline{Q}$ be the morphism of quivers defined by
\begin{center}
    $\pi((n,i))=i$ for $n \in \Z$ and $i \in Q_0$\\
    $\pi((n,\alpha))=\alpha$ for $n \in \Z$ and $\alpha \in Q_1$\\
    $\pi(\sigma(n,\alpha))=\alpha^*$ for $n \in \Z$ and $\alpha \in Q_1$
\end{center}
Let $\Z$ act on $\Z Q$ by \[ t(n,i)=(t+n,i), \quad t(n,\alpha)=(n + t, \alpha),\quad t\sigma(n,\alpha)=\sigma(t+n,\alpha).\]
Clearly $\pi$ induces bijections 

\begin{equation}\label{bijection1}
    (\Z Q)_0/\Z \stackrel{\sim}{\longrightarrow}\overline{Q}_0 \end{equation}\begin{equation}\label{bijection2}
    (\Z Q)_1/\Z \stackrel{\sim}{\longrightarrow}\overline{Q}_1
 \end{equation}

Let us still denote by $\pi$ the induced functors \[ \cp(\Z Q) \to \cp\overline{Q}\quad \text{ and }\quad k\cp(\Z Q) \to k\cp\overline{Q}\]
Clearly, $\pi$ takes a mesh relation $m_{(n,i)}=\sum_{\alpha \text{ s.t. } t(\alpha)=(n,i)} \alpha \circ \sigma(\alpha)$
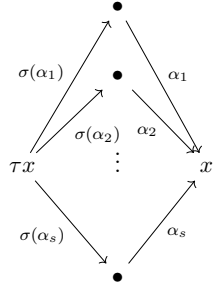
\[\adjustbox{scale=0.8}{
\begin{tikzcd}
                                                                                                        & \bullet \arrow[rdd, "\alpha_1"]  &   \\
                                                                                                        & \bullet \arrow[rd, "\alpha_2"']  &   \\
\tau x \arrow[rdd, "\sigma(\alpha_s)"'] \arrow[ruu, "\sigma(\alpha_1)"] \arrow[ru, "\sigma(\alpha_2)"'] & \vdots                           & x \\
                                                                                                        &                                  &   \\
                                                                                                        & \bullet \arrow[ruu, "\alpha_s"'] &  
\end{tikzcd}}\]
ending at $(n,i) \in (\Z Q)_0$ to the projective relation \[ \rho=\sum_{t(\alpha)=i} \alpha \alpha^*.\]
Thus, $\pi$ induces a functor \[\pi: k(\Z Q) \to k\cp\overline{Q}/\langle \rho \rangle =:\tilde{\Lambda_Q}\] from the mesh category to the preprojective category, \ie the category with object set $Q_0$ whose associated algebra is the preprojective algebra. Thanks to the bijections \ref{bijection1} and \ref{bijection2}, the functor $\pi$ induces isomorphisms \[\bigoplus_{p \in \Z} \Hom_{k(\Z Q)}((0,i),\tau^{-p}(0,j)) \stackrel{\sim}{\longrightarrow} e_j\Lambda_Qe_i\] for $i,j \in Q_0$. By Happel's Theorem, we obtain isomorphisms
\[\bigoplus_{p \in \Z} \Hom_{D^b(kQ)}(P_i,\tau^{-p}P_j) \stackrel{\sim}{\longrightarrow} e_j\Lambda_Qe_i.\] Finally, passing from categories to algebras, we get the required isomorphism \[\Lambda_Q^{ho} \stackrel{\sim}{\longrightarrow}\Lambda_Q^{co}.\]
  
 \textit{Second step: Isomorphism between $\Lambda^{ho}_Q$ and $\Lambda^{te}_Q$.}     Let $A=kQ$. We want to construct an isomorphism \[ T_A(\Omega)\simeq\bigoplus_{i,j \in Q_0}\bigoplus_{p\ge0}\Hom_{D^b(A)}(P_i, \tau^{-p}P_j) .\]
    Thus, we have to construct an isomorphism \[\Hom_{D^b(A)}(P_i, \tau^{-p}P_j)\simeq e_j\Omega^{\otimes_{A}p}e_i.\]
    We can note that $e_j\Omega^{\otimes_{A} p}e_i$ is isomorphic to $\Hom(e_iA, e_jA\otimes_{A}\Omega^{\otimes_{A} p}).$ Then, since the $\Hom$ functor commutes with finite coproducts, it suffices to show that \[\Hom_{D^b(A)}(A, \tau^{-p}A)\simeq \Omega^{\otimes_{A} p}.\]
    By definition, we have  $\Hom_{D^b(A)}(A,X)\simeq H^0(X)$, for $X \in D^b(A)$. Therefore, we have to show that \[ H^0(\tau^{-p}A)\simeq \Hom_{D^b( A)}(A,\tau^{-p}A)\simeq \Omega^{\otimes_{A} p}.\]
    By \Cref{S-1} the inverse of the Serre Functor in $D^b(A)$ is \[-\otimes_A^\mathrm{L} \mathrm{R}\Hom_{A}(D(A),A).\]
    If we denote $\mathrm{R}\Hom_{A}(D(A),A)$ by $(DA)^{\vee}$, then we have  \[\tau^{-p}A=A \otimes_A^\mathrm{L} ((DA)^\vee[1])^{\otimes^\mL_A p}=((DA)^\vee[1])^{\otimes^\mL_A p},\]
    where the first equality comes from \[\tau^-(A\otimes_A^\mL (DA)^{\vee}[1])=A \otimes^\mL_A (DA)^\vee[1] \otimes^\mL_A (DA)^\vee[1].\]
    Then we have \[H^0(\tau^{-p}A)=H^0(((DA)^{\vee}[1])^{\otimes^\mL_A p})= (H^0((DA)^{\vee}[1]))^{\otimes_A p},\]
    where the last equality is due to the fact that the homology of $(DA)^{\vee}[1]$, is concentrated in degrees $\le 0$ so that we can apply the following lemma.
    \begin{lemma}\label{keylemma}
    Consider a complex of $A$-$A$-bimodules $X \in D(\mod A^e)$ such that $H^n(X)=0$ for any $n > 0$. Then we have \[H^0(X^{\otimes^\mL_A p})=(H^0X)^{\otimes_A p}.\]
    \end{lemma}
    We have the following canonical isomorphisms
    \[H^0((DA)^\vee[1])=H^1((DA)^\vee)=H^1(\mathrm{R}\Hom_{A}(D(A),A))=\Ext^1_A(DA,A).\]
    Finally, putting all the results together, we conclude \[\Hom_{D^b(A)}(A, \tau^{-p}A)=H^0(\tau^{-p}A)=\]
    \[=(H^0((DA)^\vee[1]))^{\otimes_A p}=(\Ext^1_A(DA,A))^{\otimes_A p}=\Omega^{\otimes_A p}.\]
    \end{proof}
    
\begin{remark}
    The results in this article actually hold in a more general setting. In particular, it is possible to prove them for all finite and acyclic quivers $Q$. Under these weaker hypotheses, we can deduce from Happel's results in \cite{Happel88}, that the mesh category $k(\Z Q)$ is isomorphic to a full subcategory of the derived category $\ind D^b(kQ)$ . Then, we can generalize the proof of the \Cref{equivalence} using this characterization. Since the aim of this article is to give a reference for the results and the techniques in the proofs, we are not generalising them here. 
\end{remark}

 \def\cprime{$'$} \def\cprime{$'$}
\providecommand{\bysame}{\leavevmode\hbox to3em{\hrulefill}\thinspace}
\providecommand{\MR}{\relax\ifhmode\unskip\space\fi MR }
\providecommand{\MRhref}[2]{%
  \href{http://www.ams.org/mathscinet-getitem?mr=#1}{#2}
}
\providecommand{\href}[2]{#2}

\end{document}